\documentclass[12pt]{article}

\usepackage{amstext    }
\usepackage{amsthm    }
\usepackage{a4}
\usepackage[mathscr]{eucal}
\usepackage{mathrsfs}
\usepackage{hyperref}

\usepackage{amsmath}
\usepackage{amssymb}
\usepackage{amscd}

\newtheorem{theorem}{Theorem}[section]
\newtheorem{definition}[theorem]{Definition}
\newtheorem{proposition}[theorem]{Proposition}
\newtheorem{corollary}[theorem]{Corollary}
\newtheorem{lemma}[theorem]{Lemma}
\newtheorem{remark}[theorem]{Remark}

\newcommand{\cali}[1]{\mathscr{#1}}

\newcommand{\dist}{{\rm dist}}

\newcommand{\ddbar}{{\partial\overline\partial}}

\newcommand{\lof}{\mathop{\mathrm{{log^\star}}}\nolimits}

\newcommand{\Cc}{\cali{C}}

\newcommand{\Fc}{\cali{F}}

\newcommand{\Lc}{\cali{L}}

\newcommand{\C}{\mathbb{C}}
\newcommand{\D}{\mathbb{D}}

\newcommand{\R}{\mathbb{R}}
\newcommand{\T}{\mathbb{T}}
\newcommand{\B}{\mathbb{B}}
\newcommand{\U}{\mathbb{U}}

\renewcommand{\P}{\mathbb{P}}

\title{Directed harmonic  currents near  hyperbolic singularities}

\author{ Vi{\^e}t-Anh Nguy{\^e}n}

\begin{document}

\maketitle

\begin{abstract}
 Let $\Fc$ be a  holomorphic  foliation by curves defined in a neighborhood of $0$ in $\C^2$
 having  $0$ as  a hyperbolic  singularity.
 Let $T$ be a  harmonic current  directed  by $\Fc$  which 
    does not give  mass  to any of the two  separatrices. 
Then we  show that the Lelong number  of $T$ at $0$ vanishes.
Next, we    apply this  local result   to investigate  the global   mass-distribution for directed  harmonic  currents on  singular holomorphic foliations  living on  compact  complex   surfaces.   
 Finally, we  apply  this global result  to study the  recurrence phenomenon of a generic  leaf.    \end{abstract}

\noindent
{\bf Classification AMS 2010:} Primary: 37F75,  37A30;  Secondary: 57R30.



 \section{Introduction} \label{Intro}


While investigating  the  unique  ergodicity of harmonic currents  on singular holomorphic foliations in $\P^2,$ 
Forn\ae ss and Sibony in \cite[Corollary  2]{FornaessSibony3} have  established,  among other things, the following    remarkable  result.
  \begin{theorem}\label{thm_FS}{\rm (Forn\ae ss-Sibony \cite{FornaessSibony3})}
 Let $(M,\Fc,E)$    be  
a      hyperbolic foliation   with     the set of  singularities $E$ in a    compact complex surface $M.$ 
Assume that all the singularities  are hyperbolic and  that the foliation  has no invariant analytic  curve.
Then for every   harmonic  current $T$  directed  by $\Fc,$  its transverse measure is   diffuse,
that is,  $T$   gives no mass to each single  leaf.
 \end{theorem}
 In fact, the  original  version  of   Forn\ae ss-Sibony theorem  is only  formulated   for the case $M=\P^2.$
 However,
 their argument  still goes through (at least)  in the above general   context. On the  other hand,
 a  convenient  way to quantify   the density  of harmonic  currents  is to use  the
 notion of Lelong number introduced by Skoda \cite{Skoda}.
 Indeed, Theorem \ref{thm_FS} is    equivalent to  the  assertion that 
 the   Lelong  number of $T$ vanishes everywhere \underline{outside}    $E.$ 
Complementarily to  this  theorem,   
 the main  purpose  of the   present work  is   to investigate  
the mass-clustering  phenomenon  of  $T$    \underline{near} the set of singularities $E.$ 
Here is  our main  result which is  of local nature. 

\begin{theorem}\label{thm_main} {\rm (Main Theorem).} 
 Let $(\D^2,\Fc,\{0\})$    be  
a     holomorphic  foliation   on the unit bidisc $\D^2$  associated  to the linear vector field $\Phi(z,w) = \mu z {\partial\over \partial z}
+ \lambda w{\partial\over \partial w},$ where  $\lambda,\mu$ are some nonzero complex numbers  such that $\lambda/\mu\not\in \R.$  
Then for every   harmonic  current $T$  directed by $\Fc$  which does  not give mass to any of the two separatrices $(z=0)$ and $(w=0),$    the   Lelong  number of $T$ at $0$  vanishes.  
 \end{theorem}
 Note that  the hypothesis on the linear vector field means that $0$ is  an isolated  hyperbolic  singularity of the foliation
(see, for example, the recent survey \cite{FornaessSibony2}).
In order to prove Theorem \ref{thm_main}  we are inspired by   the approach   of Forn\ae ss--Sibony in \cite{FornaessSibony2,FornaessSibony3}
which is  based on integral formulas. Indeed,
the nature of the holonomy  maps   associated to a hyperbolic   singularity 
   permits    to use  Poisson   representation  formula
for harmonic functions on  leaves associated  to a given harmonic current  near the  singularity. 
 Therefore, we  are led to  analyze the behavior of some singular integrals at the  infinity,
i.e. when the leaves get close to the  separatrices.
Using delicate estimates on Poisson kernel,   we are able to handle these singular integrals.

Combining  Theorem \ref{thm_main} and Theorem \ref{thm_FS}, we obtain the  following result which  gives   a rather complete picture
of  the mass-distribution   of directed harmonic  currents in dimension $2.$ 
 \begin{theorem}\label{thm_main_global}
Let $(M,\Fc,E)$    be  
a      hyperbolic foliation   with      the set of  singularities $E$ in a    compact complex surface $M.$ 
Assume that all the singularities  are hyperbolic and  that the foliation  has no invariant analytic  curve.
Then for every   harmonic  current $T$  directed by $\Fc,$   the   Lelong  number of $T$ vanishes everywhere   in $M.$  
 \end{theorem}

The  above  theorem  and  a result by   Glutsyuk \cite{Glutsyuk} and Lins Neto  \cite{Neto} gives us the following corollary.
It can be   applied  to  every generic  foliation in $\P^2$ with a given degree $d>1.$
  
\begin{corollary}\label{cor_th_main}
 Let $(\P^2,\Fc,E)$    be  
a   singular  foliation by Riemann surfaces on the complex projective plane $\P^2.$ Assume that
all the singularities  are hyperbolic and  that  $\Fc$ has no invariant algebraic  curve. 
Then for every   harmonic  current $T$  directed by $\Fc,$   the   Lelong  number of $T$ vanishes everywhere   in $\P^2.$ 
\end{corollary}  
It is  worthy noting that  under the  hypothesis of Corollary \ref{cor_th_main} there is  a unique harmonic  current $T$ of mass $1$
directed by $\Fc.$ Indeed, this is  a  consequence of  Forn\ae ss-Sibony theorem  on the  unique  ergodicity of harmonic  currents
(see Theorem 4 in \cite{FornaessSibony3}).

 As  an application  of  our results 
 we will  study  the problem  of  leaf recurrence. This problem  asks  
how  often   the  leaf  $L_a$ of  a point $a,$ which is generic  with respect to a directed harmonic  current $T,$ visits  the neighborhood of a given  point $x.$  Our approach    to  this question is to apply  a  geometric  Birkhoff ergodic theorem which has recently been obtained
in our joint-work 
  with Dinh and  Sibony  \cite{DinhNguyenSibony1}.   The theorem  permits  us to define, using the leafwise Poincar\'e metric, an indicator which measures
  the frequency of a generic leaf visiting  a small ball near a given point in terms of the radius of the ball. This, combined  Theorem \ref{thm_main_global}, gives us
  an upper estimate on the frequency outside and near  singularities  (see Theorem \ref{thm_visibility} below).

The article  is  organized as  follows. In   Section \ref{Background}   below  we  set up the  background 
of the article.
 Next, we develop our main  estimates in   Section  \ref {Main_estimates} which  are the core of the  work.
The proof  of  Theorem \ref{thm_main} and Theorem \ref{thm_main_global} will be  provided   in Section \ref{Proofs}.
The   recurrence phenomenon  of  a generic leaf will be  studied in  Section \ref{section_application}.
The  article  is concluded  with some  remarks and  open questions.
  
  \smallskip
  
\noindent {\bf Notes added in proof. } If, in Theorem \ref{thm_main_global}, we assume in addition that
$M$ is a projective surface, then  our recent  work  \cite[integrability condition (1.1)] {Nguyen} provides
the following  estimate:
$$
\int_X   {T\wedge \omega(x)\over  (\dist(x,E))^2 \log^*\dist(x,E)}<\infty.
$$
Here $\omega$ is a Hermitian metric on $X$ and $\dist$  is the  distance on $M$ induced  by $\omega,$
and $\log^*(\cdot):=1+|\log(\cdot)|$ is a log-type function.
   This   inequality is  much  more  difficult to obtain  than  the vanishing  of Lelong number  of $T$  at  singularities
   established in    Theorem \ref{thm_main_global}.
 

 \section{Background} \label{Background}

 
Let $M$ be a complex   surface.  A {\it   holomorphic foliation (by Riemann surfaces)}   $(M,\Fc)$ on $M$  is  the  data of  a {\it foliation  atlas} with charts 
$$\Phi_p:\U_p\rightarrow \B_p\times \T_p.$$
Here, $\T_p$ and $\B_p$  are  domains in $\C,$   $\U_p$ is  a  domain in 
$M,$ and  
$\Phi_p$ is  biholomorphic,  and  all the changes of coordinates $\Phi_p\circ\Phi_q^{-1}$ are of the form
$$x=(y,t)\mapsto x'=(y',t'), \quad y'=\Psi(y,t),\quad t'=\Lambda(t).$$

The open set $\U_p$ is called a {\it flow
  box} and the Riemann surface $\Phi_p^{-1}\{t=c\}$ in $\U_p$ with $c\in\T_p$ is a {\it
  plaque}. The property of the above coordinate changes insures that
the plaques in different flow boxes are compatible in the intersection of
the boxes. Two plaques are {\it adjacent} if they have non-empty intersection.
 
A {\it leaf} $L$ is a minimal connected subset of $M$ such
that if $L$ intersects a plaque, it contains that plaque. So a leaf $L$
is a  Riemann surface  immersed in $M$ which is a
union of plaques. A leaf through a point $x$ of this foliation is often denoted by $L_x.$
  A {\it transversal} is  a Riemann surface  immersed in  $X$  which is  transverse to  the leaves  of $\Fc.$

A {\it    holomorphic foliation   with singularities} is  the data
$(M,\Fc,E)$, where $M$ is a complex   surface, $E$ a closed
subset of $M$ and $(M\setminus E,\Fc)$ is a  holomorphic  foliation. Each point in $E$ is  said to be  a {\it  singular point,} and   $E$ is said to be {\it the  set of singularities} of the foliation. We always  assume that $\overline{M\setminus E}=M$, see  e.g. \cite{DinhNguyenSibony1, FornaessSibony1,FornaessSibony2}  for more details.
 A leaf $L$  of the  foliation is  said to be  {\it hyperbolic} if
it  is a   hyperbolic  Riemann  surface, i.e., it is  uniformized   by 
the unit disc $\D.$   
  The   foliation   is  said to be {\it hyperbolic} if  
  its leaves   are all  hyperbolic.

  Consider a   holomorphic foliation    $(M,\Fc,E)$ with  a discrete set of 
singularities $E$ on a complex surface $M.$ We say that a
singular point $x\in E$ is {\it linearizable} if 
there is a (local) holomorphic coordinates system of $M$ on an open neighborhood $\U_x$ of $x$ on which
$(\U_x,x)$ is identified with  $(\D^2,0)$ and 
 the
leaves of $(M,\Fc,E)$ are, under this  identification, integral curves of a  linear vector field
$\Phi(z,w) = \mu z {\partial\over \partial z}
+ \lambda w{\partial\over \partial w}$  with some nonzero complex numbers $\lambda,\mu.$   
Such   neighborhood  $\U_x$ is  called 
a {\it singular flow box} of  $x.$  Moreover, we say that a
linearizable singular point $x\in E$ is {\it hyperbolic} if 
  $\lambda/\mu\not\in\R.$ 

Let $\Cc_\Fc$ (resp. $\Cc^{1,1}_\Fc$) denote the  space   of functions   (resp. forms of bidegree $(1,1)$) defined  on
leaves  of the foliations and  compactly supported  on $M\setminus E$ which are  leafwise  smooth
 and transversally continuous.  A form $\alpha \in \Cc^{1,1}_\Fc$ is  said to be {\it positive} if its restriction to every plaque
 is  a  positive $(1,1)$-form in the  usual  sense.
 \begin{definition}\label{D:harmonic_current}\rm
A {\it   harmonic current $T$   directed by the foliation $\Fc$} (or equivalently, a {\it  directed harmonic current $T$ on  $\Fc$})   is  a  linear continuous  form
    on $\Cc^{1,1}_\Fc$ which  verifies the following two conditions:
 \begin{enumerate}
\item [(i)]
$i\ddbar T=0$ in the  weak sense, that is,  $T(i\ddbar f)=0$ for all $f\in  \Cc_\Fc,$  where in the  expression $i\ddbar f,$
we only consider  $\ddbar$ along the leaves;   
\item[(ii)] 
$T$ is  positive, that is, $T(\alpha)\geq 0$ for all positive forms $\alpha\in \Cc^{1,1}_\Fc$. 
\end{enumerate}
\end{definition}

Let $\U\simeq \B\times\T$ be a  flow box. By identifying  $\T$  with a fiber of the natural projection
of $\U$ onto $\B,$
 we may regard  $\T$ as a transversal.
Then a harmonic  current $T$ in $\U$  can be decomposed  as 
\begin{equation}\label{eq_local_description_intro}
T=\int_{\alpha\in\T} h_\alpha[P_\alpha] d\nu(\alpha),
\end{equation}
where, $\nu$ is  a positive measure on $\T,$  and for $\nu$-almost every $\alpha\in\T,$  
  $P_\alpha $ is  the plaque  in $\U$ passing through $\alpha$ and 
  $h_\alpha$ denotes the harmonic function associated to the current $T$ on  $P_\alpha.$ 

Recall from Skoda \cite{Skoda} that  the Lelong number  of $T$ at  a  point $x\in M$ is
 \begin{equation}\label{eq_Lelong}
 \mathcal L(T,x):= \lim_{r\to 0+}{1\over \pi r^2}\int_{B(x,r)} T\wedge i \ddbar\|y\|^2,
 \end{equation}
 where we identify, through a biholomorphic change of coordinates, a neighborhood
of $x$ in $M$ with an open neighborhood  of $0$ in  $\C^2,$ and $B(x,r)$ is thus identified with 
the Euclidean ball  with center $0$ and radius $r.$  In fact, the Lelong number $\mathcal L(T,x)$ is  independent of 
the choice of local  coordinates.  
  The reader can find a more general notion (Dinh-Sibony cohomology classes) recently introduced and studied in \cite{DinhSibony}.

 In this  work the letters $c,$ $c',$ $c'',$ $c_0,$  $c_1,$ $c_2$ etc. denote  positive constants, not necessarily the same at each  occurrence.  The notation $\gtrsim$ and $\lesssim$ means inequalities  up to  a  multiplicative constant, whereas  we  write  $\approx$ when  both inequalities  are satisfied.


 \section{Main estimates} \label{Main_estimates}


 We keep  the hypotheses of Theorem \ref{thm_main}.
 Suppose without loss of generality that
the foliation $\Fc$ is  defined on the bidisc of radius $2,$ i.e, $(2\D)^2$
in place  of $\D^2$  and that the constant $\mu$ is  equal to $1.$ Let $\Lc$ be  the foliation in $\C^2$  associated to the vector field $\Phi(z,w) = z {\partial\over \partial z}
+ \lambda w{\partial\over \partial w}$  with some complex number $\lambda = a + ib,$ $b \not= 0.$ So  $\Lc=\Fc$  on $(2\D)^2.$ 
Note that
if we flip $z$ and $w,$ we replace $\lambda$ by $1/\lambda = \bar\lambda/|\lambda|^2 = a/(a^2 + b^2) - ib/(a^2 + b^2).$  Therefore, we
may assume without loss of generality  that  $b > 0.$ 
   We now
describe a general leaf of $\Lc$.
There are two separatrices, $(w = 0),$ $ (z = 0).$ Other than that a leaf $L$
of $\Lc$  is  equal  to $L_{(1,\alpha)}=:L_\alpha$ for some $\alpha\in\C\setminus\{0\}.$  
Following  \cite[Section 2]{FornaessSibony3}
$L_\alpha$ can be
parametrized by 
\begin{equation}\label{eq_parametrization_DS}
(z,w) = \psi_\alpha(\zeta),\ z = e^{i(\zeta+(\log {|\alpha|})/b)},\ \zeta = u+iv,\ w = \alpha e^{i\lambda(\zeta+(\log {|\alpha|})/b)},
\end{equation}
because  $\psi_\alpha( -\log {|\alpha|}/b)=(1,\alpha).$
 Setting $t:=bu+av,$ we have that 
\begin{equation}\label{eq_u,v_vs_z,w}
|z| = e^{-v}, |w| = e^{-bu-av}=e^{-t}.
\end{equation}
Observe that as we follow $z$ once counterclockwise around the origin, $u$ increases
by $2\pi$, so the absolute value of $|w|$ decreases by the multiplicative factor of $e^{-2\pi b}.$
Hence, we cover all leaves of $\Fc|_{\D^2}$ by restricting the values of $\alpha$ so that $e^{-2\pi b }= |\alpha| < 1.$ We
notice that with the above parametrization, the intersection with the unit bidisc $\D^2$
of the leaf is given by  the  domain $\{ (u,v)\in\R^2:\ v > 0, u > -av/b\}.$ 
The main   point of this special parametrization is that  the above domain  is  independent of $\alpha.$ In the $(u, v)$-plane this domain
corresponds to a sector $ S_\lambda$ with corner at $0$ and given by $0 < \theta < \arctan(-b/a)$
where the $\arctan$ is chosen to have values in $(0, \pi).$ Let $\gamma := {\pi\over
\arctan(-b/a)} .$  It is  important to note that  $\gamma > 1.$  Then the
map 
\begin{equation}\label{eq_u,v_vs_U,V}
\phi : \tau=u+iv \mapsto \tau^\gamma=(u+iv)^\gamma =:U+iV
\end{equation}
 maps this sector to the upper half plane with coordinates $(U, V ).$

The local leaf clusters on both separatrices. To investigate the clustering on the
$z$-axis, we use a transversal $\T_{z_0} := \{(z_0,w): |w| < 1\}$ for some $|z_0| = 1.$ We
can normalize so that $h_\alpha(z_0,w) = 1$ where $(z_0,w)$ is the point on the local leaf
with $e^{-2\pi b} \leq |w| < 1.$ So $(z_0,w) = \psi_\alpha(\zeta_0) = \psi_\alpha(u_0 + iv_0)$ with $v_0 = 0$ and
$0 < u_0 \leq 2\pi$ determined by the equations $|z_0| = e^{-v_0} = 1$ and $e^{-2\pi b} \leq |w| =
e^{-bu_0-av_0} < 1.$ 
Let $T$ be  a  harmonic  current of mass $1$  directed by $\Fc.$ 
Let $\U$ be a  flow box which   admits  $\T_{z_0}$ as a transversal.
Then by  (\ref{eq_local_description_intro}) we can write in $\U$
\begin{equation}\label{eq_local_description}
T=\int h_\alpha[P_\alpha] d\nu(\alpha),
\end{equation}
where, for $\nu$-almost  every  $\alpha$ satisfying $ e^{-2\pi b}\leq |\alpha|\leq 1,$   $h_\alpha$ denotes the harmonic function associated to the current $T$ on the plaque  $P_\alpha $ which is  contained  in the leaf $L_\alpha.$ We still denote by $h_\alpha$
its harmonic continuation along $L_\alpha.$
 Define $\tilde h_\alpha(\zeta) :=h_\alpha \left (
e^{i(\zeta+(\log |\alpha|)/b)}, \alpha e^{i\lambda (\zeta+(\log |\alpha|)/b)}\right)$
on $S_\lambda.$  Consider the harmonic  function  $\tilde H_\alpha:=\tilde h_\alpha\circ \phi^{-1}$ defined  in the upper half plane  $\{  U+i V:\  V>0\},$
where $\phi$ is given in (\ref{eq_u,v_vs_U,V}).  Recall  the following result from \cite{FornaessSibony3}.
\begin{lemma} \label{lem_FS}
 The harmonic  function $\tilde H_\alpha $ is the Poisson integral of its boundary values.
So in the upper half plane  $\{  U+iV:\  V>0\}$,
$$
\tilde H_\alpha(U+iV)={1\over\pi} \int_{-\infty}^\infty \tilde H_\alpha( y){V\over V^2+( y-U)^2} d y
$$ for $\nu$-almost every $\alpha.$ Moreover,
$$
\int_{e^{-2\pi b}\leq |\alpha|\leq 1}\int_{-\infty}^\infty\tilde H_\alpha( y)(1+ | y|)^{1/\gamma -1}d y d\nu(\alpha)<\infty.
$$
\end{lemma}
\begin{proof}
The lemma  follows  from  Proposition 1  and Remark 1 in \cite{FornaessSibony3}.
The finiteness of the  integral   follows  from the finiteness of the
total  mass of the  harmonic  currents on the  disjoint flow boxes  crossed when  we follow
a path  around   the two separatrices, but away from  the singularity $0.$   
 \end{proof}

For $0<r<1,$ 
let
\begin{equation}\label{eq_F}
F(r):=  \int_{B_r}  T\wedge i \ddbar \|x\|^2,
\end{equation} 
where $B_r$ denotes the ball  centered at $0$ with radius $r$ in $\C^2.$ 
Consider also the function
\begin{equation}\label{eq_G}
G(r):=
{1\over r^2}F(r).
\end{equation}
By Skoda  \cite{Skoda}, $G(r)$  decreases as  $r\searrow 0,$  and  $\lim_{r\to 0} G(r)$ is the Lelong number $\mathcal L(T,0)$ of $T$ at $0.$
On the other hand, for each $s>0$ consider  two domains 
$$D^*_s:=\{(u,v)\in S_\lambda: \min\{v,bu+av\}\geq  s\}\ \text{and}\  D_s:=\{ (t,v)\in\R^2:\  \min\{t,v\}\geq  s \},$$ 
 and the  function $K_s:\   \R\to\R^+$ given  by
\begin{equation}\label{eq_K_s}
K_s(y):= \int_{D^*_s}{e^{2s-2 \min\{v, bu+av\}  }V\over V^2+( y-U)^2} dudv={1\over b} \int_{D_s}{e^{2s-2 \min\{v, t\}  }V\over V^2+( y-U)^2} dtdv,\qquad y\in \R.
\end{equation}
Here the last equality holds since  $t=bu+av$  by (\ref{eq_u,v_vs_z,w}).

  In what follows the letters $c,$ $c',$  $c_1,$ $c_2$ etc. denote  positive constants, not necessarily the same at each  occurrence.
  For two positive-valued  functions  $A$ and $B,$  we  write $A\approx B$ if there is a constant $c$ such that $c^{-1}A\leq B\leq cA.$ 
\begin{lemma}\label{lem_estimate_G}
There is a constant $c>0$  such that for every $0<r<1,$  we have 
$$
G(r)\leq c\int_{e^{-2\pi b}\leq |\alpha|\leq 1}\Big(\int_{-\infty}^\infty 
K_{-\log r}(y)\tilde H_\alpha( y)
d y \Big)d\nu(\alpha).
 $$
\end{lemma}  
\begin{proof}
Using (\ref{eq_local_description}), (\ref{eq_F}) and the parametrization (\ref{eq_parametrization_DS}), and the  assumption  that $T$ does not give mass to any
of the two separatrices  $(z=0)$ and  $(w=0),$   we have, for $0<r<1,$ that
 $$
F(r)= \int_{e^{-2\pi b}\leq |\alpha|\leq 1} \int_{\zeta\in S_\lambda:\  \|\psi_\alpha(\zeta)\|\leq r} h_\alpha(\psi_\alpha(\zeta))  \|\psi'_\alpha(\zeta)\|^2 i d\zeta\wedge d\bar\zeta d\nu(\alpha).
$$
On the other hand, we infer  from (\ref{eq_u,v_vs_z,w}) that $\|(z,w)\|=\|\psi_\alpha(\zeta)\|\leq r$ implies $\min\{v, bu+av\}\geq  -\log r.$
Moreover,  using (\ref{eq_parametrization_DS}) and (\ref{eq_u,v_vs_z,w}) again, we get that
 $$\|\psi'_\alpha(\zeta)\|=\sqrt{|z|^2+|\lambda w|^2}\leq  (1+|\lambda|)e^{-\min\{v, bu+av\}}.
$$
Consequently,
 $$
F(r)\leq  c\int_{e^{-2\pi b}\leq |\alpha|\leq 1} \int_{(u,v)\in D^*_{-\log r}} h_\alpha(\psi_\alpha(u+iv))e^{-2\min\{v, bu+av\}}    d udv d\nu(\alpha).
$$
  Writing $U+iV=(u+iv)^\gamma$  as in (\ref{eq_u,v_vs_U,V}), an application of  Lemma \ref{lem_FS}  yields
  that
  $$
  h_\alpha(\psi_\alpha(u+iv))={1\over\pi} \int_{-\infty}^\infty \tilde H_\alpha( y){V\over V^2+( y-U)^2} d y
$$ for $\nu$-almost every $\alpha.$ 
  Inserting this  into the last estimate for $F(r)$ and taking (\ref{eq_G}) into account and  writing $r^{-2}=e^{-2\log r},$ the lemma follows.
\end{proof}

The  next  lemma  studies  the behavior of the Poisson kernel ${V\over V^2+( y-U)^2}$ in terms of $u$ and  $v.$ 
\begin{lemma}\label{lem_Poisson_kernel}
There are constants $c_1, c_2,c_3 >1$ large enough  with $c_3>c_2$ such that the following properties hold
for all $(u,v)\in\R^2$ with $\min\{v, bu+av\}\geq 1.$
\\
1)  $$  {1\over c_1}\leq {(\max\{v, bu+av\})^\gamma\over \sqrt{V^2+U^2}}\leq c_1\ \text{and}\   {1\over c_1}\leq {(\max\{v, bu+av\})^{\gamma-1}\min\{v, bu+av\} \over V}\leq c_1 .  $$
\\
2)  If  $\max\{v, bu+av\}\geq  c_2(1+|y|)^{1/\gamma},$ then
$$
{1\over c_1}  {\min\{v, bu+av\}\over (\max\{v, bu+av\})^{\gamma+1}}\leq {V\over V^2+( y-U)^2}\leq  c_1  {\min\{v, bu+av\}\over (\max\{v, bu+av\})^{\gamma+1}}.
$$
3)  If $\max\{v, bu+av\}\leq  c_2^{-1}(1+|y|)^{1/\gamma},$ 
then
$$
{1\over c_1}  {  V\over (1+|y|)^2}\leq {V\over V^2+( y-U)^2}\leq  c_1  { V\over (1+|y|)^2}.
$$
4)  If $ c_2^{-1}   (1+|y|)^{1/\gamma}\leq v, bu+av\leq  c_2(1+|y|)^{1/\gamma},$
then
$${1\over c_1}  { 1\over (1+|y|)}\leq {V\over V^2+( y-U)^2}\leq  c_1  { 1\over (1+|y|)}.$$
5)  If $ \min\{v, bu+av\}\leq c_3^{-1}   (1+|y|)^{1/\gamma}$ and  
$c_2^{-1}   (1+|y|)^{1/\gamma}\leq  \max\{v, bu+av\}  \leq  c_2(1+|y|)^{1/\gamma},$
then
$$
 {1\over c_1}\leq {V\over V^2+( y-U)^2}: {(1+|y|)^{1/\gamma-1}\min\{v, bu+av\}   \over (\min\{v, bu+av\} )^2+  (\max\{v, bu+av\} -\rho )^2 } \leq  c_1  ,
$$
where   $\rho$ is a  real number which depends only on $y$ and  $ \min\{v, bu+av\},$ and
which satisfies 
$ c_2^{-1}   (1+|y|)^{1/\gamma}\leq \rho\leq  c_2(1+|y|)^{1/\gamma}.$
\end{lemma}
\begin{proof}
\noindent{\it Proof of Part 1).}  The  first inequality of Part 1) follows  from the equality $|U+iV|=|u+iv|^\gamma.$
To prove the  second  inequality of Part 1) we use  some  elementary  trigonometric arguments.
Let $O$ denote  the origin in the $(u,v)$-plane ane let $M$ denote the point $u+iv.$
Recall that  the sector $S_\lambda$ is  delimited by two rays emanating from $O$ which  correspond to 
two lines $v=0$ and  $bu+av=0.$ 
Let $A$ (resp. $B$) be the unique point lying on the ray corresponding to  $v=0$ (resp. $bu+av=0$)
such that  $OA=1$ (resp. $OB=1$).
Let $\theta:=\measuredangle \widehat{AOM}$ and  $\vartheta:= \measuredangle \widehat{MOB}.$
Then  $\theta,\vartheta\geq 0$ and $\theta+\vartheta=\arctan(-b/a)\in(0,\pi).$
A geometric  argument gives that
$$
\sin\theta=v/ OM\quad\text{and}\quad \sin\vartheta=(bu+av)/OM.
$$
Moreover,
$$     \max\{v,bu+av\}\leq OM\leq |u|+|v| \leq(1+ (1+|a|)b^{-1}) \max\{v,bu+av\}.$$
Consequently,
\begin{equation}\label{eq_sin}
\sin\theta\approx {v\over  \max\{v,bu+av\}}          \quad\text{and}\quad \sin\vartheta\approx {bu+av\over   \max\{v,bu+av\}}.
\end{equation}
Let $N$ be the point $U+iV$ in the $(U,V)$-plane.
Let $C$ (resp. $D$) be  the image of $A$ (resp. $B$) by the map $\phi:\ \tau\mapsto \tau^\gamma$ given in (\ref{eq_u,v_vs_U,V}).
Clearly, $\measuredangle \widehat{CON}=\gamma\theta$ and $\measuredangle \widehat{NOD}=\gamma\vartheta$ and  $\measuredangle \widehat{CON}+\measuredangle \widehat{NOD}=\gamma\theta+\gamma\vartheta=\pi.$
Suppose without loss of generality that $\theta \leq \vartheta,$ or equivalently  $ v\leq bu+av .$ Then $0\leq \theta\leq \pi/2$ and $ \measuredangle \widehat{CON}
=\gamma\theta\leq \pi/2.$  Combining this  with the   well-known estimate $2/\pi\leq (\sin{t})/t \leq 1$ for $0<t\leq \pi/2,$
we get that
$$
\sin(\gamma\theta)\approx  \gamma\theta\approx \gamma\sin \theta\approx  \sin\theta\approx  {v\over  \max\{v,bu+av\}}   ,
$$
where the last estimate  holds by (\ref{eq_sin}).
 On the other hand,  a geometric argument shows that
$$  V=  \sqrt{U^2+V^2}\cdot \sin \measuredangle \widehat{CON}= \sqrt{U^2+V^2}\cdot \sin(\gamma\theta).$$
 This, coupled  with the last estimate for $
\sin(\gamma\theta)$ and  the first  estimate  (for  $\sqrt{U^2+V^2}$)  in Part 1),
implies the second estimate  of this part.

\noindent{\it Proof of Part 2).} 
We will show that there is a constant $c>1$ such that
\begin{equation}\label{eq1_Part_2}
 c^{-1}( U^2+V^2)\leq V^2+( y-U)^2\leq c ( U^2+V^2).
\end{equation}
Taking (\ref{eq1_Part_2}) for granted, Part 2) follows  from   combining (\ref{eq1_Part_2}) with  Part 1).

Now  we  turn to the proof of (\ref{eq1_Part_2}). Using the first estimate of Part 1) and the assumption of Part 2), we have that
\begin{equation}\label{eq2_Part_2}
\sqrt{U^2+V^2}\geq c_1^{-1}(\max\{v, bu+av\})^\gamma\geq   c_1^{-1}c_2^\gamma(1+|y|).
\end{equation}
Therefore,
$$
 V^2+( y-U)^2\leq   V^2+2U^2+2y^2\leq (2+2c^2_1c_2^{-2\gamma})(U^2+V^2), 
$$
which proves  the right-side  estimate of  (\ref{eq1_Part_2}) for $c:=2+2c^2_1c_2^{-2\gamma}.$

To prove the left-side estimate of  (\ref{eq1_Part_2}), consider two cases.
If $V\geq |U|$ then  $
 V^2+( y-U)^2\geq  V^2\geq  1/2(U^2+V^2).$
If $V<|U|$  then  for $c_2>1$ large  enough, (\ref{eq2_Part_2}) yields that
$|U|\geq 2 |y|,$ which  in turn  implies that 
$
V^2+( y-U)^2\geq  V^2 +1/4 U^2\geq  1/4 (U^2+V^2).$ This completes the proof of  (\ref{eq1_Part_2}).

\noindent{\it Proof of Part 3).} 
 Using the first estimate of Part 1) and the assumption of Part 3), we have that
$$
\sqrt{U^2+V^2}\leq c_1(\max\{v, bu+av\})^\gamma\leq   c_1c_2^{-\gamma}(1+|y|).
$$
We fix $c_2>1$ is  large  enough so that the  last  line  gives  $|y|\geq   2c_1\cdot\max\{|U|,V\}.$
This  gives, using the first estimate of Part 1), that
$$
|y|\geq    c_1\sqrt{U^2+V^2}\geq (\max\{v, bu+av\})^\gamma\geq 1.
$$
Consequently, we  get, using $|y|>2|U|,$  that
$$
1/12 (1+|y|)^2\leq 1/4y^2\leq V^2+( y-U)^2\leq V^2+2U^2+2y^2\leq 4(1+|y|)^2,
$$
which completes Part 3).

\noindent{\it Proof of Part 4).}
By the assumption of Part 4), $v\approx bu+av.$
Consequently, we deduce from 
 (\ref{eq_sin}) that $\theta,\vartheta\approx 1,$ which in turn implies that  $V,U\approx \sqrt{U^2+V^2}.$
This, combined with the assumption of Part 4) and  the first estimate of Part 1), yields that
$$
V,U, \sqrt{U^2+V^2}\approx 1+|y|.
$$    
Using this  and the inequalities
$$
(1+|y|)^2\approx V^2\leq V^2+( y-U)^2\leq V^2+2U^2+2y^2\approx(1+|y|)^2,
$$
Part 4) follows.

\noindent{\it Proof of Part 5).} 
Let  $(u,v)$ as in the  assumption of Part 5) and let $c_1,c_2>1$ be  constants given by Part 1).
Suppose without loss of generality that
$v\leq t,$ where $t:= bu+av$ as usual.  Fix $c_3\geq c_2$   large  enough so that for  every $1\leq v\leq c_3^{-1}(1+|y|)^{1/\gamma},$
there exist  a    solution  $u:=u(y,v)$ of the following   equation
$$  U=y,\qquad\text{where}\  U+iV'=(u+iv)^\gamma$$
which satisfies
   $$ c_2^{-1}(1+|y|)^{1/\gamma}\leq u(y,v),\rho(y,v)\leq c_2(1+|y|)^{1/\gamma},$$
  where $\rho(y,v):=bu(y,v)+av.$ 
    Let $\rho=\rho(y,v) .$ 
Note that
$$
\min\{v,t\}\approx \min\{v,\rho\}\quad\text{and}\quad \max\{v,t\}\approx \max\{v,\rho\}\approx (1+|y|)^{1/\gamma}.
$$
Therefore,  we deduce from the second  inequality of Part 1)
that 
\begin{equation}\label{e:V-V'}
V\approx V'. 
\end{equation}
Note also  that   for $c_3\geq c_2$   large  enough,
    $u=b^{-1}(\rho-av)\approx  (1+|y|)^{1/\gamma}$ and  
 $v\ll  (1+|y|)^{1/\gamma}.$ In particular, we get that
 \begin{equation}\label{e:arg}
0<\arg(u+iv),\arg(u(y,v)+iv)<{\pi\over 2\gamma},
\end{equation}
where $\arg$  denotes the argument of a  nonzero complex number. 
 We need   the following elementary   fact.
\begin{lemma}\label{L:gamma}
Let $c',\gamma >1$ be two constants. Then there is a constant $c''$ such that
for all $w, w'\in\C\setminus\{0\}$  satisfying  
$$ c'^{-1}\leq |w'|/|w|\leq c'\quad\text{and}\quad  0\leq\arg w, \arg w'<{\pi\over 2\gamma},$$
we  have that
$$c''^{-1}|w-w'|(|w|+|w'|)^{\gamma-1}\leq |w^\gamma-w'^\gamma|\leq  c''|w-w'|(|w|+|w'|)^{\gamma-1}.$$ 
\end{lemma}
\proof Suppose  without loss of generality that
$0<\arg w\leq\arg w'<{\pi\over 2\gamma}.$ We  consider  two cases.
\\
{\bf Case  1:} $|w-w'|\geq {1\over 2}\min\{ |w|,|w'|\}.$

In this case it is  easy to show that $|w^\gamma-w'^\gamma|\approx (|w|+|w'|)^{\gamma-1}.$
\\
{\bf Case  2:} $|w-w'|\leq {1\over 2}\min\{ |w|,|w'|\}.$

 Let $w''$ be the complex  number such that $|w''|=|w|$ and $\arg w''=\arg w'$ that is,
 $w'':=|w|e^{i\arg w'}.$
 It is  not difficult to  show  in this case that $|w^\gamma-w'^\gamma|\approx |w''^\gamma-w'^\gamma|.$
 So it remains to estimate $|w''^\gamma-w'^\gamma|.$
 Since  $\arg w'=\arg w'',$ we  can reduce  the estimate to the case where  $w',w''>0$
 by multiplying both $w'$ and $w''$ by $e^{-i\arg w'}.$
 The lemma is  then an immediate consequence of
   the following elementary  inequality
\begin{equation*}
   \gamma|w'-w''|(\min\{w',w''\})^{\gamma-1}\leq |w'^\gamma-w''^\gamma|\leq  \gamma|w'-w''|(\max\{w',w''\})^{\gamma-1},\qquad  w',w''>0.
   \end{equation*}  
\endproof
Now we come back the proof of Part 5).
Recall   estimate \eqref{e:arg} and  the following  equations 
$$U+iV=(u+iv)^\gamma\quad\text{and}\quad
    y+iV'=(u(y,v)+iv)^\gamma.$$ 
Next, we deduce from \eqref{e:V-V'} that
 $ V\gtrsim |V-V'|.$
Consequently, applying Lemma \ref{L:gamma} to $w:=u+iv$ and  $w':=u(y,v)+iv,$
  yields that
\begin{eqnarray*}
|V|+|U-y|&\approx & V+  |( U+iV)-(y+iV')|\\
 &=& V+ |( u+iv)^\gamma-  (u(y,v)+iv)^\gamma| \\
& \approx & V+ |u-u(y,v)|(u+u(y,v))^{\gamma-1}
\\
&\approx &
V+ |u-u(y,v)|(1+|y|)^{1-1/\gamma}\\
&\approx &
V+ |(bu+av)-(bu(y,v)+av)|(1+|y|)^{1-1/\gamma}\\
&= & V+ |t-\rho(y,v)|(1+|y|)^{1-1/\gamma}, 
\end{eqnarray*}
  This, combined  with the second  inequality of Part 1), implies    Part 5). 
\end{proof}

The following elementary estimate is  needed.
\begin{lemma}\label{lem_exp} For  every $s_0\geq 1,$  $\int_{s_0}^\infty s e^{2s_0-2s}ds =s_0/2+1/4.$  
\end{lemma}
\begin{proof} An integration by parts gives that
$$
1/2=\int_{s_0}^\infty e^{2s_0-2s}ds= \lbrack se^{2s_0-2s}\rbrack^{s=\infty}_{s=s_0}+2\int_{s_0}^\infty se^{2s_0-2s}ds=2\int_{s_0}^\infty se^{2s_0-2s}ds-s_0.
$$
\end{proof}
Now  we arrive  at the  main estimate of this   section, i.e, a precise  behavior of  $K_s(y)$
when  the leaves  get close to the  separatrices. 
\begin{proposition}\label{prop_main_estimate_1} There is a constant $c>0$ such that 
for all $s>0$ and $y\in \R,$ 
$$
{K_s(y)\over  (1+ | y|)^{1/\gamma -1}  }\leq  c \min\left\lbrace 1,  \left\lbrack {(1+|y|)^{1/\gamma}\over  s}\right\rbrack^{\gamma-1} \right\rbrace   .
$$
\end{proposition}
\begin{proof} Let  $c_2,c_3$ be the constants with $c_3>c_2>1$ given by Lemma \ref{lem_Poisson_kernel}.
 We consider three cases.

\noindent {\bf Case 1:} $  s\geq c_2   (1+ | y|)^{1/\gamma} .$ 

By  Part 2) of Lemma \ref{lem_Poisson_kernel} and by formula (\ref{eq_K_s}),
we have that
$$
K_s(y)\leq c\int_{D_s} e^{2s-2 \min\{v, t\}  } {\min\{v, t\}\over (\max\{v, t\})^{\gamma+1}}dtdv
\leq c'\Big( \int_{t=s}^\infty  {dt\over t^{\gamma+1}}    \Big) \Big(\int_{v=s}^\infty  v e^{2s-2v}dv \Big) .
$$
The first  integral on the right hand side  is equal to
$  \gamma^{-1}s^{-\gamma},$ while the second  one is, by
Lemma \ref{lem_exp}, equal to $s/2+1/4.$  
   Hence, $
K_s(y)\leq   c s^{1-\gamma}.$ This  completes  Case 1.

\noindent {\bf Case 2:} $  c_2^{-1} \leq  {s\over   (1+ | y|)^{1/\gamma}}\leq c_2.$ 

Write $D_s= D^{'}_s\cup D^{''}_s,$ where 
\begin{eqnarray*}
  D^{'}_s&:=&\left\lbrace (t,v)\in D_s:\ \max\{t,v\}\leq  c_2 (1+ | y|)^{1/\gamma}\right\rbrace,\\
  D^{''}_s&:=&\left\lbrace (t,v):\ s\leq \min\{ t,v\}\ \text{and}\   c_2 (1+ | y|)^{1/\gamma}\leq  \max\{ t,v\} \right\rbrace.
\end{eqnarray*}
Consequently, formula (\ref{eq_K_s}) gives that
\begin{equation}\label{eq_I_II}
K_s(y)=  {1\over b} \Big(\int_{D^{'}_s}  +     \int_{D^{''}_s }\Big){e^{2s-2 \min\{v, t\}  }V\over V^2+( y-U)^2} dtdv=: I+II.
\end{equation}
To estimate $I,$  we  apply  Part 4) of Lemma \ref{lem_Poisson_kernel} and obtain that
$$
I \leq c \int_{D^{'}_s}e^{2s-2 \min\{v, t\}  }  { dtdv\over (1+|y|)}.
$$
The integral is  bounded  by a constant times
$$
 \Big( \int_{c_2^{-1}(1+ | y|)^{1/\gamma}}^{c_2(1+ | y|)^{1/\gamma}}  {dt\over (1+|y|)} \Big)\Big(\int_{v=c_2^{-1}(1+ | y|)^{1/\gamma}}^\infty    e^{2s-2v}dv  \Big). 
$$
The left integral is equal to
$  (c_2+c_2^{-1})(1+|y|)^{1/\gamma-1},$ while the  right  integral is bounded by $1/2.$
Hence, $I\leq c(1+|y|)^{1/\gamma-1}.$

To estimate $II,$  we  apply  Part 2) of Lemma \ref{lem_Poisson_kernel} and obtain that
$$
II\leq c\int_{D^{''}_s}e^{2s-2 \min\{v, t\}  }  { \min\{ t,v\}dtdv\over  (  \max\{ t,v\} )^{\gamma+1} }. 
$$  
The integral  in the last line is  smaller than a constant times
$$
\Big( \int_{t=c_2^{-1}(1+ | y|)^{1/\gamma}}^\infty  {dt\over  t^{\gamma+1}}    \Big)\Big(\int_{v=c_2^{-1}(1+ | y|)^{1/\gamma}}^\infty  v e^{2s-2v}dv\Big).
$$
The left integral is equal to 
$  \gamma^{-1}c_2^{-\gamma}(1+|y|)^{-1},$  while  the  right integral is, by
Lemma \ref{lem_exp},  equal to
$  (c_2^{-1}/2)(1+|y|)^{1/\gamma}+1/4.$ Hence, $II\leq c(1+|y|)^{1/\gamma-1}.$

 Inserting the above estimates for $I$ and $II$ into  (\ref{eq_I_II}), we obtain the  desired estimate for $K_s(y)$ in the  second case.

\noindent {\bf Case 3:} $   s  \leq  c_2^{-1}     (1+ | y|)^{1/\gamma}.$ 

 Write  
 $D_s= D^1_s\cup D^2_s\cup D^3_s,$ where
\begin{eqnarray*}
D^1_s&:=&\left\lbrace (t,v):\ s\leq t,v\leq  c^{-1}_2 (1+ | y|)^{1/\gamma}\right\rbrace,\\
D^2_s&:=&\left\lbrace (t,v):\ s\leq \min\{ t,v\}\ \text{and}\   c_2 (1+ | y|)^{1/\gamma}\leq  \max\{ t,v\} \right\rbrace,\\
D^3_s&:=&\left\lbrace (t,v):\ \max\{s,c^{-1}_3 (1+ | y|)^{1/\gamma}\}  \leq \min\{ t,v\}\  \right.\\
 &\quad& \left.\ \text{and}\  c^{-1}_2 (1+ | y|)^{1/\gamma} \leq \max\{ t,v\} \leq  c_2 (1+ | y|)^{1/\gamma}  \right\rbrace,\\
D^4_s&:=&\left\lbrace (t,v):\ s\leq \min\{ t,v\}\leq  c^{-1}_3 (1+ | y|)^{1/\gamma}  \right.\\
 &\quad &\left. \text{and}\
c^{-1}_2 (1+ | y|)^{1/\gamma}\leq \max\{ t,v\} \leq  c_2 (1+ | y|)^{1/\gamma}  \right\rbrace.
\end{eqnarray*}
 Consequently, we get,  similarly as in (\ref{eq_I_II}), that 
$$ K_s(y)=  {1\over b}\Big( \int_{D^1_s}+   \int_{D^2_s}+\int_{D^3_s}+\int_{D^4_s}\Big) {e^{2s-2 \min\{v, t\}  }V\over V^2+( y-U)^2} dtdv=: I+II+III+IV.
$$
 To estimate $I,$   we apply  Part 1) and Part 3) of Lemma \ref{lem_Poisson_kernel}. Consequently,
we  obtain that
$$
I \leq c \int_{D^1_s}  (\max\{v, t\})^{\gamma-1}\min\{v, t\} e^{2s-2 \min\{v, t\}  }  { dtdv\over (1+|y|)^2}.
$$
The integral is  bounded  by a constant times
$$
\Big( \int_{t=s}^{c^{-1}_2(1+ | y|)^{1/\gamma}}  {t^{\gamma-1}dt\over (1+|y|)^2}    \Big)\Big (\int_{v=s}^{c^{-1}_2(1+ | y|)^{1/\gamma}}  v e^{2s-2v}dv\Big). 
$$
The left integral is bounded by
$  \gamma^{-1}c_2^{-\gamma}(1+|y|)^{-1},$ while the  right integral is, 
by Lemma \ref{lem_exp}, bounded by   $s/2+1/4.$
Hence,  $  I\leq cs(1+|y|)^{-1 }.$

To estimate $II,$   we  apply  Part 2) of Lemma \ref{lem_Poisson_kernel} and obtain that
$$
II\leq c\int_{D^2_s}e^{2s-2 \min\{v, t\}  }  { \min\{ t,v\}dtdv\over  (  \max\{ t,v\})^{\gamma+1} }. 
$$  
The integral  in the last line is  smaller than a constant times
$$
\Big( \int_{t=c_2(1+ | y|)^{1/\gamma}}^\infty  {dt\over t^{\gamma+1}}    \Big) \Big (\int_{v=s}^\infty  ve^{2s-2v}dv\Big).
$$
The left integral is equal to 
$  \gamma^{-1}c_2^{-\gamma}(1+|y|)^{-1},$ while the right integral is, by
Lemma \ref{lem_exp}, equal to 
$  s/2+1/4.$ Hence, $II\leq cs(1+|y|)^{ -1}.$

To  estimate $III,$ we   apply  Part 4) of Lemma \ref{lem_Poisson_kernel} and argue as in Case  2.
Consequently, we can  show that $III\leq  c(1+|y|)^{1/\gamma-1}.$

To estimate $IV,$   we  apply  Part 5) of Lemma \ref{lem_Poisson_kernel} and obtain that
$$
IV\leq c\int_{D^4_s}e^{2s-2 \min\{v, t\}  }  { (1+|y|)^{1/\gamma-1}\min\{v, bu+av\} dtdv \over (\min\{v, bu+av\} )^2+  (\max\{v, bu+av\} -\rho )^2 } . 
$$ 
We infer from this estimate  that
$$
IV\leq c\int_{v=s}^\infty \Big(\int_{t= c^{-1}_2(1+ | y|)^{1/\gamma}}^{c_2(1+ | y|)^{1/\gamma}}
  { (1+|y|)^{1/\gamma-1} vdt \over v^2+  (t-\rho(y,v) )^2 }\Big)e^{2s-2 v }dv,
$$
where $\rho(y,v)$ satisfies  $c^{-1}_2(1+ | y|)^{1/\gamma} \leq \rho(y,v)\leq c_2(1+ | y|)^{1/\gamma}.$ 
The inner integral is bounded by $IV_1+IV_2,$ where
$$ IV_1=\int_{ |t- \rho(y,v)|\leq  v}
        { (1+|y|)^{1/\gamma-1}v dt \over v^2+  (t-\rho(y,v) )^2 }\leq 
\int_{ |t- \rho(y,v)|\leq  v}
        { (1+|y|)^{1/\gamma-1} dt \over v  }\leq
c(1+|y|)^{1/\gamma-1},
$$
 and
 $$
IV_2\leq  \int { (1+|y|)^{1/\gamma-1} vdt \over v^2+  (t-\rho(y,v) )^2 }\leq 
 \int { (1+|y|)^{1/\gamma-1} vdt \over   (t-\rho(y,v) )^2 }\leq
c(1+|y|)^{1/\gamma-1}.
$$
Here  the integrals in the last line  are taken over the region
$$\left\lbrace t\in\R:\ c^{-1}_2(1+ | y|)^{1/\gamma}\leq t\leq c_2(1+ | y|)^{1/\gamma} \ \text{and}\  |t- \rho(y,v)|\geq  v \right\rbrace.$$
 So   the inner integral $\leq  c (1+|y|)^{1/\gamma -1}.$ Hence, $IV\leq c(1+|y|)^{1/\gamma -1}.$

Combining the estimates for  $I,$  $II,$  $III$ and $IV,$ and using the  assumption
$   s  \leq  c_2^{-1}     (1+ | y|)^{1/\gamma},$ we  infer that 
$$K_s(y)=I+II+III+IV\leq c's(1+|y|)^{ -1}+c'(1+|y|)^{1/\gamma-1}\leq c (1+|y|)^{1/\gamma-1}.$$
The proof of Case 3, and hence the proposition, is  thereby completed.
\end{proof}


 \section{Proofs of the main results} \label{Proofs}


\noindent  {\bf End of the proof of  Theorem  \ref{thm_main}.}
By  Proposition  \ref{prop_main_estimate_1} the family  of functions 
 $(g_s)_{s>0}:\ \R\to\R^+,$  where $g_s$ is  given by  
$$ g_s( y):= {K_s(y)\over  (1+ | y|)^{1/\gamma -1}  },\qquad y\in\R $$
 is  uniformly bounded.
Moreover, $\lim_{s\to\infty} g_s(y)=0$ for $y\in \R.$

On the other hand, consider  the  measure $\chi$  on $\R,$ given by
$$
\int_\R \varphi d\chi= 
\int_{e^{-2\pi b}\leq |\alpha|\leq 1}\Big (\int_{-\infty}^\infty \varphi(y)\tilde H_\alpha( y)(1+ | y|)^{1/\gamma -1}d y\Big) d\nu(\alpha)
$$
for every continuous bounded test function $\varphi$ on $\R.$ 
By  Lemma \ref{lem_FS}, $\chi$ is a  finite positive measure.
 Consequently,  we get, by  dominated convergence, that 
$
\lim_{s\to\infty}\int_\R g_sd\chi=0. 
$ 
 This, combined  with  Lemma \ref{lem_estimate_G}, implies that
$$
0\leq \lim_{r\to 0+}  G(r)\leq c\cdot \lim_{s\to\infty}\int_\R g_sd\chi=  0,
$$
which,  coupled  with (\ref{eq_F})-(\ref{eq_G}), gives that
 $\mathcal L( T,0)=0,$ as  desired.
 \hfill $\square$
 
\noindent  {\bf End of the proof of  Theorem  \ref{thm_main_global}.}
Let $x\in M$ be  a point. Consider two cases.
\\
\noindent  {\bf Case 1:} $x\not\in E.$

Let $\U$ be a regular flow box with transversal $\T$ which   contains $x.$
By  (\ref{eq_local_description_intro})   we can write in $\U$ 
$$
T=\int h_\alpha[V_\alpha] d\nu(\alpha),
$$
where, for $\nu$-almost  every $\alpha\in \T,$   $h_\alpha$ denote the positive harmonic function associated to the current $T$ on the plaque  $V_\alpha. $ By Harnack inequality, there is a constant $c>0$ independent of $\alpha$ such that
 $$  c^{-1}h_\alpha(z)\leq  h_\alpha(w)\leq    ch_\alpha(z),\qquad  z,w\in V_\alpha.$$
 Using this and  the above  local description of $T$ on $\U$ and formula (\ref{eq_Lelong}), we  infer easily that $\mathcal L(T,x)\leq c  \nu(\{x\}).$
 On the  other hand, by  Theorem \ref{thm_FS}, $\nu(\{x\})=0.$ Consequently,   $\mathcal L(T,x)=0.$
 
\noindent  {\bf Case 2:} $x\in E.$

  Fix a (local) holomorphic coordinates system of $M$ on a singular  flow box $\U_x$ of $x$  such that
$(\U_x,x)$ is identified with  $(\D^2,0)$ and 
 the
leaves of $(M,\Fc,E)$ are integral curves of the  linear vector field
$\Phi(z,w) = \mu z {\partial\over \partial z}
+ \lambda w{\partial\over \partial w}$  with some nonzero complex numbers $\lambda,\mu$   
  such that $\lambda/\mu\not\in \R.$  
On the  other hand, it follows from   Theorem \ref{thm_FS} that  $T$   gives no mass to each single  leaf.
 In particular, $T$  does  not give mass to any of the two separatrices $(z=0)$ and $(w=0).$
 Consequently, we  are able to apply   Theorem \ref{thm_main}. Hence,      $\mathcal L(T,x)=0.$  
  \hfill $\square$


\section{Application: recurrence of generic  leaves}
\label{section_application}


Let $(M,\Fc,E)$    be  
a      hyperbolic foliation   with     the set of  singularities $E$ in a   Hermitian compact complex surface $(M,\omega).$ 
Let $\dist$ be the  distance on $M$ induced by the Hermitian metric.
Assume that all the singularities  are hyperbolic and  that the foliation  has no invariant analytic  curve.
Let $T$ be a nonzero
directed harmonic current on $(X,\Lc,E)$. The  existence  of    such a current      has been  established   by Berndtsson-Sibony in \cite[Theorem 1.4]{BerndtssonSibony}, and   Forn\ae ss-Sibony   in \cite[Corollary 3]{FornaessSibony2}.
Assume  in addition that $T$ is   extremal (in the convex  set of all directed harmonic currents).
Let $\omega_P$ be the Poincar\'e metric on $\D,$ given by
$$ \omega_P(\zeta):={2\over (1-|\zeta|^2)^2} i d\zeta\wedge d\overline\zeta,\qquad\zeta\in\D.  $$
For any point $a\in M\setminus E$  consider a universal covering map
$\phi_a:\D\rightarrow L_a$ such that $\phi_a(0)=a$. This map is
uniquely defined by $a$ up to a rotation on $\D$. 
Then, by pushing   forward  the Poincar\'e metric $\omega_P$
on $\D$  
  via $\phi_a,$ we obtain the  so-called {\it Poincar\'e metric} on $L_a$ which depends only on the leaf.  The latter metric is given by a positive $(1,1)$-form on $L_a$  that we also denote by $\omega_P$ for the sake of simplicity.
  Since     the measure  $m_P:=T\wedge\omega_P$ is,   by  \cite{DinhNguyenSibony1},  of  finite mass, 
 we may
assume  without loss of generality that $m_P$ 
 is a probability measure. So, $m_P$ is a harmonic measure on $X$ with respect to $\omega_P$.

In this  section we   study the following problem.
Given    a point $x\in M$ and  a  $m_P$-generic  point $a\in M\setminus  E,$
how  often  does the  leaf  $L_a$ visit  the ball $B(x,r)$   as  $r\searrow 0$ ?  
Here  $B(x,r)$ (resp. $\overline {B}(x,r)$) denotes the open (resp. closed) ball with  center $x$ and radius $r$ with respect to the metric $\dist.$
The purpose of
 this  section is to 
 apply  Theorem \ref{thm_main_global}  in order to   obtain a partial answer to this   question.

 Let us  introduce  some more  notation and terminology.
Denote by
$r\D$ the disc of center $0$ and of radius $r$ with $0<r<1$. In the
Poincar{\'e} disc $(\D,\omega_P),$ 
$r\D$ is also the disc of center 0 and of radius 
$$R:=\log{1+r\over 1-r}\cdot$$
So, we will also denote by $\D_R$ this disc and by $\partial \D_R$ its  boundary.

Together with Dinh and Sibony,  we introduce   the following  indicator.  
\begin{definition}\label{defi_visibility}\rm For  each $r>0,$ 
the {\it visibility  of  a point $a\in M\setminus E$  within   distance $r$  from a  point $x\in M$} is the number
$$N(a,x,r)= \limsup_{R\to \infty}\frac{1}{R}\int\limits_0^R \Big( \int_{\theta=0}^1  \textbf{1}_{ B(x,r)}\big (\phi_a(s_te^{2\pi i\theta})\big)d\theta\Big) dt 
 \in [0,1],$$
where $ \textbf{1}_{ B(x,r)}$ is
the  characteristic  function  associated to the set $B(x,r),$ and $s_t$  is  defined by  the  relation $t=\log{1+s_t\over 1-s_t}\cdot,$ that is,  $s_t\D=\D_t .$
\end{definition}
Geometrically, $N(a,x,r)$  is the  average, as $R\to\infty,$  over  the hyperbolic time $t\in [0, R]$ of  the Lebesgue measure of the set
$\{  \theta\in[0,1]:\  \phi_a(s_te^{2\pi i\theta})  \in B(x,r)\}.$ 
The last  quantity  may be interpreted  as the portion  which hits  $B(x,r)$ of the Poincar\'e circle of radius $t$  with center $a$ spanned on the leaf $L_a.$  

We will see in  Lemma \ref{lem_relation_N_vs_m_P} that  the  $\limsup$ in Definition \ref{defi_visibility} above  can be replaced by  a true limit
for $m_P$-almost  every $a\in M\setminus E.$  
Moreover, Definition \ref{defi_visibility} can also be  applied to  singular holomorphic foliations (by hyperbolic Riemann surfaces)
in arbitrary dimensions.

Now  we  are in the position to state the main result of this  section.
\begin{theorem}\label{thm_visibility}
 We keep  the  above  hypothesis and notation.  Then  for $m_P$-almost  every  point $a\in M\setminus E$
 and for every point $x\in M,$  we have that
 $$
 N(a,x,r)=\begin{cases}
 o( r^2), & x\in M\setminus E;\\
 o(|\log r|^{-1}), & x\in E.
 \end{cases}
 $$
\end{theorem}

For  the proof of this  theorem  we need   some more preparatory results.
For all $0<R<\infty$, consider the following   measure on $M$: 
$$m_{a,R}:=\frac{1}{M_R}(\phi_a)_* \big( (\log\frac{1}{|\zeta|}\omega_P)|_{\D_R}\big).$$
where 
$\omega_P$ denotes also the Poincar{\'e} metric on $\D$ and 
$$M_R:= \int_{\D_R} \log \frac{1}{|\zeta|}\omega_P=\int_{\zeta\in \D_R} \log \frac{1}{|\zeta|} \frac{2}{(1-|\zeta|^2)^2}
id\zeta\wedge d\overline\zeta.$$
So, $m_{a,R}$ is a probability measure  which depends on $a,R$ but 
does not depend on the choice of $\phi_a$. 
Recall the following  geometric Birkhoff ergodic  theorem. 

\begin{theorem}\label{th_birkhoff}{\rm (Dinh-Nguyen-Sibony \cite{DinhNguyenSibony1})}
 Under the  above  hypothesis and notation,
then for almost every point $a\in X$ with respect to the measure
$m_P$, the measure $m_{a,R}$ defined above converges to $m_P$
when $R\to\infty$.  
\end{theorem}

The  above theorem gives the  following  connection between $N(a,x,r)$ and $m_P.$

\begin{lemma}\label{lem_relation_N_vs_m_P} For $m_P$-almost  every $a\in M\setminus E$ and for all $x\in M,$   the  $\limsup$ in Definition \ref{defi_visibility} is  in fact  a  true limit. 
Moreover, if $m_P(\partial B(x,r))=0,$ then 
$$  
N(a,x,r)= \lim_{R\to \infty} m_{a,R}(B(x,r))= m_P(B(x,r)) .
$$
\end{lemma}
Notice that there  is  a number $r_0>0$ small enough such that for every
$x\in M$ and every  $0<r<r_0,$ we have that   $m_P(\partial B(x,r))=0.$ 
\begin{proof}
Let $l_R$ be the length  in the Poincar\'e metric  of the  circle $\partial \D_R.$
For   a     continuous  test  function $\varphi$  on $M,$  
 we have that
$$ \frac{1}{R}\int\limits_0^R \Big( \int_{\theta=0}^1  \varphi \big (\phi_a(s_te^{2\pi i\theta})\big)d\theta\Big) dt = \frac{1}{R}\int\limits_0^R 
 \Big (  \int_{ (\phi_a)_* [\partial\D_t]  } {\varphi\cdot d\sqrt{\omega_P}\over l_t}\Big) dt,$$
 where  $d\sqrt{\omega_P}$ is the  length element associated to the metric $\omega_P.$
Moreover, using the polar coordinates,   we   get that
\begin{multline*} \Big|\frac{1}{R}\int\limits_0^R 
 \Big (  \int_{ (\phi_a)_* [\partial\D_t]  } {\varphi\cdot  d\sqrt{\omega_P}\over l_t}\Big) dt -
\int\varphi\cdot \frac{1}{ 2\pi R}(\phi_a)_* \big( (\log\frac{1}{|\zeta|}\omega_P)|_{\D_R}\big)\Big |\\
\leq
{\|\varphi\|_\infty\over R}\int_{t=0}^R  |l_t (2\pi)^{-1}\log (1/s_t)-1|  dt.
\end{multline*}
Since  $|l_t (2\pi)^{-1}\log (1/s_t)-1|\approx  e^{-t},$  the right hand side   tends to $0$ as $R\to\infty.$

Next, a direct  computation shows that  $|M_R-2\pi R|$ is  bounded by  a  constant.
Consequently,
$$
\int  \varphi\cdot \frac{1}{ 2\pi R}(\phi_a)_* \big( (\log \frac{1}{|\zeta|}\omega_P)|_{\D_R}\big) -  \int_M  \varphi m_{a,R}
$$
tends to $0$ as $R\to\infty.$
On the other hand, we  infer from  Theorem \ref{th_birkhoff}  that
$\lim_{R\to \infty} \int\varphi m_{a,R}= \int\varphi m_P $ for $m_P$-almost  every $a\in M\setminus E.$
Putting   these  estimates altogether, we   obtain that
$$ \lim_{R\to \infty}\frac{1}{R}\int\limits_0^R \Big( \int_{\theta=0}^1   \varphi\big (\phi_a(s_te^{2\pi i\theta})\big)d\theta\Big) dt = \int_M\varphi m_P .$$
Writing    $\textbf{1}_{B(x,r)}$ (resp.    $\textbf{1}_{\overline{B}(x,r)}$)  as  the  
limit of an  increasing (resp. decreasing)   sequence of  continuous  test functions $\varphi$ and using that  $m_P(\partial B(x,r))=0,$
the lemma follows from  the last equality.
\end{proof}

 For simplicity  we still denote by   $\omega$  the  Hermitian  metric on leaves of the foliation $(M\setminus E,\Fc)$  induced by  the ambient Hermitian metric  $\omega.$
  Consider the  function $\eta:\ M\setminus E\to [0,\infty]$ defined by
  $$
  \eta(x):=\sup\left\lbrace \|D\phi(0)\| :\ \phi:\ \D\to L_x\ \text{holomorphic such that}\ \phi(0)=x  \right\rbrace.
  $$
  Here, for the  norm of the  differential $D\phi$ we use the Poincar\'e metric on $\D$ and the Hermitian metric
  $\omega$ on $L_x.$
   We obtain the following relation  between  $\omega$   and  the Poincar\'e metric $\omega_P$ on leaves
\begin{equation}
\label{eq_relation_Poincare_Hermitian_metrics}
\omega=\eta^2 \omega_P.
\end{equation}

We record here  the  following   precise estimate on the  function $\eta.$  
\begin{lemma} \label{lem_poincare}
We keep  the above  hypotheses
and notation. Then there exists  a  constant $c>1$  such that  $\eta \leq c$ on $M$, $\eta\geq c^{-1}$ outside the singular flow boxes $\cup_{x\in E}{1\over 4}\U_x$ and 
$$c^{-1} s \lof s \leq\eta(x)  \leq c  s \lof s$$
for $x\in M\setminus E$  and $s:=\dist(x,E).$ Here $\lof(\cdot):=1+|\log(\cdot)|$ is a log-type function,
and ${1\over 4}\U_x:= (1/4\D)^2$ for $\U_x=\D^2.$ 
\end{lemma}
\begin{proof}
Since  there exists no  holomorphic non-constant map
$\C\rightarrow M$  such that out of  $E$ the image of $\C$ is  locally contained in leaves,
it follows from 
 \cite[Theorem 15]{FornaessSibony2} 
that there is  a constant  $c>0$  such that
$\eta(x) \leq  c$ 
for all $x\in M\setminus E.$ In other  words,
 the   foliation   is   {\it Brody hyperbolic} 
following  the terminology of  our  joint-work with Dinh and Sibony   \cite{DinhNguyenSibony3}.
 Therefore,  the lemma is  a  direct consequence  of  Proposition 3.3 in \cite{DinhNguyenSibony3}.
\end{proof}

Now  we arrive at the
\\
\noindent{\bf End of the proof of  Theorem  \ref{thm_visibility}.}    
   Let $x\in M$ be  a point. Consider two cases.
\\
\noindent  {\bf Case 1:} $x\not\in E.$

Let $\U$ be a regular flow box with transversal $\T$ which   contains $x.$
By  (\ref{eq_local_description_intro})   we can write in $\U$ 
$$
T=\int h_\alpha[P_\alpha] d\nu(\alpha),
$$
where, for $\nu$-almost  every $\alpha\in \T,$   $h_\alpha$ denotes the positive harmonic function associated to the current $T$ on the plaque  $P_\alpha. $ On the other hand,  since  $\U$  is  away  from  the set of singularities $E,$
we deduce  from  Lemma \ref{lem_poincare} that  $c^{-1}\leq  \eta(y)\leq  c$ for $y\in\U.$
 Using this and (\ref{eq_relation_Poincare_Hermitian_metrics}) and the above  expression for $T,$
we  infer easily that 
 $$
 m_P(y)=(T\wedge \omega_P)(y)= \eta(y) (T\wedge \omega)(y)\approx  (T\wedge \omega)(y)\approx  T\wedge i\ddbar \| y\|^2,\qquad y\in\U.  
 $$
 This, combined  with  formula (\ref{eq_Lelong}), implies that
 $$
 \lim_{r\to 0} r^{-2}m_P(B(x,r))\leq  \lim_{r\to 0} cr^{-2} \int_{B(x,r)} T\wedge i\ddbar \| y\|^2 =c \mathcal L(T,x).
 $$
  By  Theorem \ref{thm_FS}, $\mathcal L (T,x)=0.$ 
 On the other hand, we know from 
Lemma  \ref{lem_relation_N_vs_m_P}  that
$$  
\lim_{r\to 0} r^{-2} N(a,x,r)=  \lim_{r\to 0} r^{-2} m_P(B(x,r))
$$
for $m_P$-almost  every $a\in M\setminus E.$ 
 Putting the last three  estimates together, we  obtain that   $ 
 N(a,x,r)= 
 o( r^2).$
 
\noindent  {\bf Case 2:} $x\in E.$

  Fix a (local) holomorphic coordinates system of $M$ on a singular  flow box $\U_x$ of $x$  such that
$(\U_x,x)$ is identified with  $(\D^2,0)$ and 
 the
leaves of $(M,\Fc,E)$ are integral curves of the  linear vector field
$\Phi(z,w) = \mu z {\partial\over \partial z}
+ \lambda w{\partial\over \partial w}$  with some nonzero complex numbers $\lambda,\mu$   
  such that $\lambda/\mu\not\in \R.$  
  
  Suppose without loss of generality that the metric $\omega$ coincides with the standard 
 metric $i\ddbar \| y\|^2$ on $\D^2.$
 Next, recall from 
(\ref{eq_relation_Poincare_Hermitian_metrics}) that
$$
i\ddbar \|y\|^2=\eta^2(y) g_P(y)\approx  \|y\|^2 (\log\|y\|)^2 g_P(y)\quad\text{for}\ 0<\|y\|<1/2.
$$
where  the  estimate $\approx$  holds by  Lemma  \ref{lem_poincare}.  
Therefore, we infer that 
$$m_P(y):= (T\wedge  \omega_P)(y)\approx  { T\wedge i\ddbar \|y\|^2\over  \|y\|^2 (\log\|y\|)^2}\quad\text{on}\  B_{1/2}.$$ 
Consequently, for $0<r<1/2,$
$$
\int_{B_r}  m_P(y)\approx \int_{B_r}{ T\wedge i\ddbar \|y\|^2\over  \|y\|^2 (\log\|y\|)^2}
=\int_0^r {F'(s)ds \over  s^2 (\log s)^2},
$$
where the last equality follows from (\ref{eq_F}). So Case  2 will follow if  we can 
show that $\int_0^r {F'(s)ds \over  s^2 (\log s)^2}=o(|\log r|^{-1})
$  as $r\to 0.$
Since we know  from (\ref{eq_F})-(\ref{eq_G}) that $(s^2G(s))'=F'(s),$ performing an integration by part to the last expression yields that
\begin{multline*}
\int_0^r {F'(s)ds \over  s^2 (\log s)^2}=\int_0^r  { d(s^2G(s)) \over  s^2 (\log s)^2}\\
 =\left \lbrack { G(s) \over   (\log s)^2}\right\rbrack^r_0
+2 \int_0^r  { G(s)ds \over s (\log s)^2}+2\int_0^r  { G(s)ds \over s (\log s)^3}.
\end{multline*}
On the  other hand, we know  from (\ref{eq_F})-(\ref{eq_G}) that $G(r)$ decreases, as $r\searrow 0,$ to $\mathcal L(T,x),$
which is  equal to $0$ by  Theorem \ref{thm_main_global}.
 Therefore,  a  straightforward  computation shows that all three terms on the right hand  side
of  the last line  is  of order  
$  o(|\log r|^{-1})
$  as $r\to 0,$  as  desired.  
 This  completes the proof of  Case 2, and hence  the theorem is  proved.
      \hfill $\square$
  
\begin{remark}\rm
We conclude  the article  with some  remarks and open questions.

\noindent 1) It seems  to be of interest to  investigate   the Main Theorem in the case  where the singularity $0$ is  only
linearizable (see \cite{DinhNguyenSibony1}).

\noindent 2) A natural question  arises   whether     the main results of this  article can be  generalized  in higher  dimensions.
We postpone this issue  to a  forthcoming work.

\noindent 3) Now let $(M,\Fc,E)$    be  
a      hyperbolic foliation   with     the set of  singularities $E$ in a   Hermitian compact complex  manifold $(M,\omega)$
of arbitrary dimension. 
 Assume that all the singularities  are linearizable. Using  the finiteness of the Lelong number  of a positive harmonic current
 \cite{Skoda}, and applying \cite{DinhNguyenSibony1}   and  arguing as in the end  of  the proof
 of Theorem \ref{thm_visibility},  we can show  the following  weak form of   this theorem (but in
higher dimension).
 For $m_P$-almost  every  point $a\in M\setminus E$
 and for every point $x\in M,$  we have that
 $$
 N(a,x,r)=\begin{cases}
 O( r^2), & x\in M\setminus E;\\
 O(|\log r|^{-1}), & x\in E.
 \end{cases}
 $$
  \end{remark}

\noindent
Vi{\^e}t-Anh Nguy{\^e}n,  
Universit\'e de Lille 1, 
Laboratoire de math\'ematiques Paul Painlev\'e, 
CNRS U.M.R. 8524,  
59655 Villeneuve d'Ascq Cedex, 
France.\\
{\tt Viet-Anh.Nguyen@math.univ-lille1.fr},
{\tt http://www.math.u-psud.fr/$\sim$vietanh}


\small

\begin{thebibliography}{99}
 \addcontentsline{toc}{section}{References}

 
\bibitem{BerndtssonSibony}  
Berndtsson, Bo; Sibony, Nessim. The $\overline\partial$-equation on a positive current.
{\it Invent. Math.} {\bf 147} (2002), no. 2, 371-428. 
    

\bibitem{DinhNguyenSibony1}
Dinh, T.-C.; Nguy\^en, V.-A.; Sibony, N.  Heat equation and ergodic theorems for Riemann surface laminations.
{\it Math. Ann.} {\bf 354}, (2012), no. 1,   331-376.

 \bibitem{DinhNguyenSibony2}
Dinh, T.-C.; Nguy\^en, V.-A.; Sibony, N.  Entropy for hyperbolic Riemann surface laminations I.
{\it Frontiers in Complex Dynamics: a volume in honor of John Milnor's 80th birthday,} (A. Bonifant, M. Lyubich, S. Sutherland, editors), (2012), Princeton University Press, 24 pages.

 \bibitem{DinhNguyenSibony3}
Dinh, T.-C.; Nguy\^en, V.-A.; Sibony, N.  Entropy for hyperbolic Riemann surface laminations II.
{\it Frontiers in Complex Dynamics: a volume in honor of John Milnor's 80th birthday,} (A. Bonifant, M. Lyubich, S. Sutherland, editors), (2012), Princeton University Press, 29 pages. 

 \bibitem{DinhSibony} Dinh, Tien-Cuong; Sibony, Nessim. Unique ergodicity for foliations in $\P^2$ with an invariant curve.
  {\tt math.CV, math.DS,  arXiv:1509.07711,} 28 pages.
 
\bibitem{FornaessSibony1}
 Forn\ae ss, John Erik; Sibony, Nessim.   Harmonic currents of finite
 energy and laminations. {\it Geom. Funct. Anal.}, {\bf  15} (2005), no. 5, 962-1003.
 
  
\bibitem{FornaessSibony2}
  Forn\ae ss, John Erik; Sibony, Nessim.   Riemann surface laminations with singularities.
{\it J. Geom. Anal.}, {\bf 18} (2008), no. 2, 400-442.

 
 
 
\bibitem{FornaessSibony3}
  Forn\ae ss, John Erik; Sibony, Nessim. Unique ergodicity of harmonic currents on singular foliations
 of $\P^2.$ {\it Geom. Funct. Anal.}, {\bf 19} (2010), no. 5,  1334-1377.
 
 
\bibitem{Glutsyuk} 
Glutsyuk, A.A.  Hyperbolicity of the leaves of a generic one-dimensional holomorphic foliation on a nonsingular projective algebraic variety. (Russian)
 {\it Tr. Mat. Inst. Steklova}, {\bf 213} (1997), Differ. Uravn. s Veshchestv. i Kompleks. Vrem., 90-111; translation in {\it Proc. Steklov Inst. Math.} 1996, no. 2, {\bf 213}, 83-103. 

 

\bibitem{Neto}
Lins Neto A.  Uniformization and the Poincar\'e metric on the leaves of a foliation by curves.
{\it  Bol. Soc. Brasil. Mat. (N.S.),}  {\bf 31} (2000), no. 3, 351-366. 

 
 
  
\bibitem{NetoSoares}
Lins Neto, A.; Soares M. G.
Algebraic solutions of one-dimensional foliations.
{\it J. Differential Geom.,}  {\bf 43} (1996), no. 3, 652-673. 
  
  \bibitem{Nguyen}
  Nguyen, Viet-Anh. 
 Singular holomorphic foliations by curves I: Integrability of holonomy cocycle in dimension 2.  
 {\tt math.DS, math.CV, math.DG arXiv:1403.7688,} 73 pages.

   
   \bibitem{Skoda}
  Skoda, Henri. Prolongement des courants, positifs, ferm\'es de masse finie. (French) [Extension of closed, positive currents of finite mass] {\it Invent. Math.} {\bf 66} (1982), no. 3, 361–376.   
 

 \end{thebibliography}
\end{document}